\documentclass[12pt,reqno]{amsart}

\usepackage{amsmath}
\usepackage{amsfonts}
\usepackage{amssymb}
\usepackage{amsthm}
\usepackage{esint}
\usepackage{mathrsfs}
\usepackage{mathtools}
\usepackage{tensor}

\usepackage{indentfirst}
\usepackage{ucs}
\usepackage[utf8x]{inputenc}

\usepackage{verbatim}
\usepackage{url}          % to make URLs work
\usepackage{hyperref}     % to make links to sections/references work
\usepackage{enumerate}    % to use better enumeration styles
\usepackage{color}        % to show comments in red
\usepackage{tikz} \usetikzlibrary{arrows,shapes} % to draw diagrams

\newcommand{\CC}{\mathbb{C}}
\newcommand{\RR}{\mathbb{R}}
\newcommand{\ZZ}{\mathbb{Z}}
\renewcommand{\SS}{\mathbb{S}}

\newcommand{\cQ}{\mathcal{Q}}

\newcommand{\sL}{\mathscr{L}}
\newcommand{\sM}{\mathscr{M}}

\DeclareMathOperator{\tr}{tr}

\DeclareMathOperator{\area}{area}
\DeclareMathOperator{\length}{length}
\DeclareMathOperator{\divg}{div}

\newcommand{\closure}[1]{\overline{#1}}

\numberwithin{equation}{section}

\theoremstyle{definition} 
\theoremstyle{definition} \newtheorem*{rema}{Remark}
\theoremstyle{definition} \newtheorem{theo}[equation]{Theorem}
\theoremstyle{definition} \newtheorem*{theo*}{Theorem}
\theoremstyle{definition} \newtheorem{prop}[equation]{Proposition}
\theoremstyle{definition} \newtheorem{coro}[equation]{Corollary}
\theoremstyle{definition} \newtheorem*{coro*}{Corollary}
\theoremstyle{definition} \newtheorem{lemm}[equation]{Lemma}
\theoremstyle{definition} \newtheorem*{fact}{Fact}
\theoremstyle{definition} 
\theoremstyle{definition} 
\theoremstyle{definition} 
\theoremstyle{definition} 

\allowdisplaybreaks

\title[On the Bartnik mass of apparent horizons]{On the Bartnik mass of apparent horizons}
\author{Christos Mantoulidis}
\address{Department of Mathematics \\
	Stanford University \\
	Stanford, CA 94305}
\email{c.mantoulidis@math.stanford.edu}
\author{Richard Schoen}
\address{Department of Mathematics \\
	University of California \\
	Irvine, CA 92617}
\thanks{The second author was partially supported by NSF grant DMS-1404966}
\email{rschoen@math.uci.edu}

\begin{document}

\maketitle

\begin{abstract}
In this paper we characterize the intrinsic geometry of apparent horizons (outermost
marginally outer trapped surfaces) in asymptotically flat spacetimes; that is, the Riemannian metrics
on the two sphere which can arise. Furthermore we determine the minimal ADM mass of a 
spacetime containing such an apparent horizon. The results are conveniently formulated
in terms of the quasi-local mass introduced by Bartnik \cite{bartnik-quasilocal-mass} in
1989. The Hawking mass provides a lower bound for Bartnik's quasilocal mass on apparent horizons by way of Penrose's conjecture on time symmetric slices, proven in 1997 by Huisken and Ilmanen \cite{huisken-ilmanen-riemannian-penrose-inequality} and in full generality in 1999 by Bray \cite{bray-riemannian-penrose-inequality}. We compute Bartnik's mass for all non-degenerate apparent horizons and show that it coincides with the Hawking mass. As a corollary we disprove a conjecture due to Gibbons in the spirit of Thorne's hoop conjecture \cite{gibbons-hoop-conjecture}, and construct a new large class of examples of apparent horizons with the
integral of the negative part of the Gauss curvature arbitrarily large.
\end{abstract}

\section*{Statement of results}

Hawking's theorem on the topology of black holes asserts that cross sections of the event horizon in four-dimensional asymptotically flat stationary black hole spacetimes obeying the dominant energy condition are topologically 2-spheres \cite{hawking-black-holes}. The result extends to apparent horizons in black hole spacetimes that are not necessarily stationary. In the time symmetric setting, which we adopt in this paper, these horizons are stable minimal spheres \cite{huisken-ilmanen-riemannian-penrose-inequality} in the initial data set $M$. If $\Sigma$ is such a surface, the second variation of area being non-negative for a normal variation with lapse function $\varphi$ is the condition
\[ \frac{1}{2}\int_\Sigma (R_M-R_\Sigma+\|S\|^2)\varphi^2\ dA\leq \int_\Sigma \|\nabla \varphi\|^2\ dA
\]
where $R_M$ and $R_\Sigma$ are the scalar curvatures of the three metric and the two metric
respectively, and $S$ denotes the second fundamental form of $\Sigma$ in $M$. Since $R_M\geq 0$
and the Gauss curvature $K$ is equal to $\frac{1}{2}R_\Sigma$, it follows that the second order linear elliptic operator $-\Delta + K$ is a non-negative
operator on the apparent horizon $\Sigma$. We denote its first eigenvalue by $\lambda_1(-\Delta+K)$, so that we have $\lambda_1(-\Delta+K)\geq 0$.

Let $\sM_+$ denote the class of metrics $g$ on $\SS^2$ for which $\lambda_1(-\Delta_g + K_g) > 0$. In Sections \ref{sec:collar.extensions} and \ref{sec:gluing.schwarzschild} we compute the Bartnik mass of all horizons in $\sM_+$ by constructing admissible extensions whose ADM mass is arbitrarily close to the optimal value which is the Hawking mass in view of the Riemannian Penrose inequality \cite{bray-riemannian-penrose-inequality}, \cite{huisken-ilmanen-riemannian-penrose-inequality}. The extension theorem is as follows.

\begin{theo*}
Let $g \in \sM_+$. For any $m > 0$ such that $16 \, \pi \, m^2 > \area(g)$ there exists an asymptotically flat 3-manifold $M^3$ with non-negative scalar curvature such that
\begin{enumerate}[(i)]
\item $\partial M^3$ is isometric to $(\SS^2, g)$ and minimal,
\item $M^3$ is isometric to a mass-$m$ Schwarzschild metric outside a compact set, and
\item $M^3$ is foliated by mean convex spheres that eventually coincide with standard Schwarzschild coordinate spheres.
\end{enumerate}
\end{theo*}

There is a diagram, Figure \ref{figure}, following Lemma \ref{collar.lemma} which illustrates the geometry
of our extension. The extension procedure is explicit and allows us to keep track of the mass accurately. In Section \ref{sec:collar.extensions} we use variational methods to construct a scalar positive collar extension for $g$. Section \ref{sec:gluing.schwarzschild} gives a gluing procedure for
attaching this collar extension to a Schwarzschild exterior region with mass arbitrarily close to the optimal value. This construction confirms the fact that symmetry at infinity, even coincidence with Schwarzschild, is not a rigid condition \cite{corvino-scalar-curvature-deformation}. More importantly, it follows that:

\begin{enumerate}[(a)]
\item For every $g \in \sM_+$ Bartnik's quasilocal mass is $m_B(\SS^2, g, H = 0) = (\area(g)/16 \pi)^{1/2}$ and is in particular a smooth function of the boundary data, i.e., the metric.
\item The rigidity case of the Riemannian Penrose inequality in the presence of an apparent horizon says that the Hawking mass of the horizon is only attained as the ADM mass in the trivial case of Schwarzschild metrics. It follows, a posteriori, that with the known exception of Schwarzschild there are no exact minimizers that attain the optimal mass $m_B(\SS^2, g, H = 0)$, $g \in \sM_+$; constructing a minimizing sequence of extensions that are precisely Schwarzschild at infinity is, in this sense, optimal.
\item While we only construct initial data sets in this paper, we can construct a local spacetime containing this data by filling the interior to construct a complete initial data set with scalar
curvature non-negative, and with the property
that there is a single apparent horizon. We can then treat the region where the scalar curvature
is positive as a fluid body and consider the Cauchy problem for the
Einstein equations coupled with a perfect fluid. We choose a state function of the form
$p=k\rho^\gamma$ with $\gamma>1$ and $k$ positive. We take the initial density to
be a multiple of the scalar curvature as dictated by the scalar constraint equation, and we 
take the initial three velocity of the fluid to be zero. We can then
use the existence theorem of A. Rendall \cite{rendall-cauchy-problem} for the Cauchy problem to construct a maximal smooth spacetime extension. We observe that in our construction
the function $\rho$ decays exponentially at the fluid boundary, so the auxiliary function $w$ introduced in \cite{rendall-cauchy-problem} is smooth with compact support. We also note that
if $k$ is chosen small enough we will have $p\leq \rho$ in a neighorhood of the initial
slice, and so the dominant energy condition is satisfied.
We expect that after a relatively short time there will be an apparent horizon which is near
a round metric on $\SS^2$ since the initial data set contains approximately round two
spheres which are nearly apparent horizons (see Figure \ref{figure}).
We do not know if there is a modification of our construction which produces vacuum initial data. 
\end{enumerate}

The methods in Sections \ref{sec:collar.extensions} and \ref{sec:gluing.schwarzschild} give a complete intrinsic characterization of all non-degenerate (in the sense of second variation of area) apparent horizons, showing that they are in bijective correspondence with metrics in $\sM_+$. This result was obtained by B. Smith  \cite{smith-black-holes-prescribed-full} who also was able to specify the second fundamental
form of the horizon. The methods of \cite{smith-black-holes-prescribed-full} do not give
control on the ADM mass of the extension as we do in this paper. The question of characterizing degenerate horizons---non-generic examples that are totally geodesic in a scalar flat region and satisfy $\lambda_1(-\Delta_g + K_g) = 0$---remains open.

The $\sM_+$ characterization is analytically complete but the possible geometry of
such metrics is harder to determine. The examples that have been known to date to arise as horizons all seem to have positive curvature. In Section \ref{sec:horizon.examples} we study the relevant differential operator $-\Delta + K$ and construct a class of horizons that allows a large amount of negative curvature. We also show that this class is dense in the appropriate topology.

\begin{theo*}
The set $\sM_+$ is relatively open in every pointwise conformal class of metrics, with respect to the uniform $C^1$ topology. Furthermore, for every $c > 0$ the subset of $\sM_+$ consisting of metrics $g$ with $\int_{\SS^2} (K_g)_- \, dA_g \geq c$ is $C^1$ dense in $\sM_+$ where $(K_g)_-=\max\{0,-K_g\}$.
\end{theo*}

Finally, in Section \ref{sec:gibbons.hoop.conjecture} we discuss Gibbons' formulation 
\cite{gibbons-hoop-conjecture} of Thorne's hoop conjecture in general relativity which states that the Birkhoff invariant $\beta(g)$ of an apparent horizon is no larger than 
$4 \, \pi \, m_{\operatorname{ADM}}$. In a weaker version of this conjecture one replaces $\beta(g)$ with the length of the shortest nontrivial geodesic, $\ell(g)$. More details on the conjectures and their precise formulation can be found in Section \ref{sec:gibbons.hoop.conjecture}. We construct initial data sets that disprove these conjectures.

\begin{theo*}
There exists a complete, asymptotically flat $(M^3, g)$ with non-negative scalar curvature, that is precisely Schwarzschild outside a compact set, whose unique horizon is totally geodesic, strictly stable, and satisfies $\ell(\partial M^3, g) > 4 \, \pi \, m_{\operatorname{ADM}}(M^3, g)$.
\end{theo*}

Our counterexample is independent of Section \ref{sec:horizon.examples}. We use a construction originally due to Croke: a smoothed out bi-equilateral triangle with positive curvature that has been conjectured by Calabi to have maximal systolic ratio \cite{croke-length-shortest-geodesic}, \cite{calabi-cao-simple-closed-geodesics}. While the question of optimality is open, the systolic ratio is large enough that coupled with our optimal mass computation in Section \ref{sec:gluing.schwarzschild} settles Gibbons' conjecture in the negative. We would like to thank Davi Maximo for pointing out Gibbons' conjecture to us and for helpful conversations during the early stages of this work.

\section{Canonical collar extensions of $\sM_+$}

\label{sec:collar.extensions}

In this section we construct scalar positive collar extensions of metrics in $\sM_+$ that we will then turn into complete initial data sets in Section \ref{sec:gluing.schwarzschild}. These collar extensions are topological cylinders $\SS^2 \times [0,1]$ foliated by strictly mean convex spheres from $\sM_+$, all of which are round near $\SS^2 \times \{1\}$. We endow the cylinders with a scalar positive metric by solving a continuous family of variational problems.

\begin{prop}
The space $\sM_+$ is path-connected.
\end{prop}
\begin{proof}
Let $g$ be such that $\lambda_1(-\Delta_g + K_g) > 0$. By uniformization we may write $g = e^{2w} \, g_*$ for a round metric $g_*$ with area $4 \, \pi$. It suffices to show that the path $t \mapsto g_t \triangleq e^{2tw} \, g_*$ stays within $\sM_+$. Let $f \in C^\infty$, $f \not \equiv 0$. The relevant bilinear form for $-\Delta_{g_t} + K_{g_t}$, $t \in [0,1]$, is
\[ \int |\nabla^{g_t} f|_{g_t}^2 + K_{g_t} \, f^2 \, dA_{g_t} = \int |\nabla^{*} f|_{*}^2 + (1 - t \, \Delta_* w) \, f^2 \, dA_{*} \text{.} \]
This quantity is positive at $t = 0$ since $g_0 \in \sM_+$ (it is round) and positive at $t = 1$ since $g_1 = g \in \sM_+$ (by assumption). It is also linear in $t$, so it is positive for all $t \in [0,1]$. Since $f$ was arbitrary, we have $\lambda_1(-\Delta_{g_t} + K_{g_t}) > 0$ for all $t \in [0, 1]$.
\end{proof}

\begin{rema}
The same argument shows that the larger set $\closure{\sM_+}$ of metrics $g$ on $\SS^2$ with $\lambda_1(-\Delta_g + K_g) \geq 0$ is also path connected.
\end{rema}

\begin{lemm}
\label{extension.lemma}
For any $g \in \sM_+$ there exists a smooth path of metrics $t \mapsto g(t) \in \sM_+$, $t \in [0,1]$, with
\begin{enumerate}[(i)]
\item $g(0) = g$,
\item $g(1)$ round,
\item $\dot{g} \equiv 0$ for $t \in [1/2,1]$, and
\item $\frac{d}{dt} dA_{g(t)} \equiv 0$ for all $t \in [0,1]$ where $dA_g$ denotes the area
form for a metric $g$.
\end{enumerate}
\end{lemm}
\begin{proof}
By uniformization we may write $g = e^{2 w(x)} g_*$ for some round metric $g_*$ with area $4 \, \pi$. Fix a smooth decreasing function $\zeta : [0,1] \to [0,1]$ with $\zeta(0) = 1$, $\zeta(1) = 0$, and $\zeta'(t) \equiv 0$ on $[1/2,1]$. We have already seen that $t \mapsto e^{2 \zeta(t) w(x)} g_*$, $t \in [0,1]$ is a smooth path in $\sM_+$ from $g$ to the round metric $g_*$. Since $\sM_+$ is invariant under dilations it follows that $t \mapsto h_t \triangleq e^{2\zeta(t)w(x)+2a(t)} g_* \in \sM_+$ too, with $a(t)$ chosen so that $a(0) = 1$ and $\frac{d}{dt} \area(\SS^2, e^{2\zeta(t)w(x)+2a(t)} g_*) = 0$; i.e.,
\[ a'(t) = - \zeta'(t) \fint_{\SS^2} w(x) \, dA_{e^{2\zeta w} g_*} \text{.} \]
Having fixed the metrics $h(t)$, for each $t$ let $X_t$ be such that $\divg_{h(t)} X_t = -2 \, (\zeta'(t) \, w+a'(t))$. Then take $\phi_t$ to be the integral flow along $X_t$ and consider $g(t) = \phi_t{}^* h(t)$. Then
\begin{align*}
\frac{d}{dt} dA_{g(t)} & = \frac{d}{dt} \phi_t{}^* dA_{h(t)} = \phi_t{}^* \Big[ \frac{d}{dt} dA_{h(t)} + \sL_{\dot{\phi}_t} dA_{h(t)} \Big] \\
    & = \phi_t{}^* \Big[ \frac{1}{2} \, \tr_{h(t)} \dot{h}(t) \, dA_{h(t)} + \divg_{h(t)} \dot{\phi}_t \, dA_{h(t)} \Big] \\
    & = \phi_t{}^* \Big[ 2 \, (a'(t) + \zeta'(t) \, w) + \divg_{h(t)} \dot{\phi}_t \Big] \, dA_{h(t)} = 0 \text{,}
\end{align*}
and the result follows.
\end{proof}

\begin{rema}
We can solve the prescribed divergence equation $\divg_{h(t)} X_t = -2(\zeta'(t) \, w+a'(t))$ smoothly in time. Indeed we can instead solve the elliptic equation $\Delta_{h(t)} \psi(t, \cdot) = -2 \, (\zeta'(t)\, w+a'(t))$ since the right hand side integrates to zero against the $h(t)$ area form, and $t \mapsto \psi(t, \cdot)$ is a smooth family of $C^\infty$ maps by elliptic regularity. Then we simply take $X_t = \nabla^{h(t)} \psi(t, \cdot)$.
\end{rema}

\begin{rema}
A consequence of $\zeta'(t) \equiv 0$ for $t \in [1/2,1]$ is that all the $g(t)$ with $t \in [1/2,1]$ are round spheres.
\end{rema}

From this point on we fix this path $t \mapsto g(t) \in \sM_+$ constructed by the lemma. For each $t$ we pick $u(t, \cdot)$ to be an eigenfunction corresponding to the first eigenvalue $\lambda(t)$ of $-\Delta_{g(t)} + K_{g(t)}$. This can be done in such a way that $(t, x) \mapsto u(t, x) > 0$ is smooth since the first eigenspace of these operators is one-dimensional. (See Appendix \ref{app:eigenfunction.smoothness}.) Normalize our choice of eigenfunctions so that, say, $u(t, \cdot)$ has unit $L^2$ norm with respect to the area form $dA_{g(t)}$. Notice then that $u(t, \cdot)$, $t \in [1/2, 1]$, depends on $\area(g(0))$ but not on $t$.

\begin{lemm}[Collar lemma]
\label{collar.lemma}
There exist $0 < \varepsilon_0 \ll 1$ and $A_0 \gg 1$ depending on $g(0)$ such that for all $\varepsilon \in (0, \varepsilon_0]$, $A \geq A_0$, the topological cylinder $\SS^2 \times [0,1]$ endowed with the metric
\[ \gamma = (1+\varepsilon \, t^2) \, g(t) + A^2 \, u(t, \cdot)^2 \, dt^2 \]
has the following properties:
\begin{enumerate}[(i)]
\item the scalar curvature is positive,
\item the initial boundary sphere $\{t = 0 \}$ is minimal, and
\item the remaining foliating spheres $\{ t = t_0 \}$, $t_0 \in (0, 1]$, are all strictly mean-convex relative to the outward normal direction (the $\partial_t$ direction).
\end{enumerate}
\end{lemm}

\begin{rema}
We can also arrange for $\{ t = 0 \}$ to be totally geodesic without any change in the proof beyond requiring that the auxiliary function $\zeta$ in Lemma \ref{extension.lemma} further satisfy $\zeta'(0) = \zeta''(0) = \ldots = 0$.
\end{rema}

\begin{proof}[Proof of Lemma \ref{collar.lemma}]
Write $h(t) = (1+ \varepsilon \ t^2) \, g(t)$, $v(t, x) = A \, u(t, x)$, and $\mu(t) = (1 + \varepsilon \, t^2)^{-1} \, \lambda(t)$ to simplify our notation to $\gamma = h(t) + v(t, x)^2 \, dt^2$ and $v(t, \cdot)$ now being an eigenfunction of $-\Delta_{h(t)} + K_{h(t)}$ with eigenvalue $\mu(t)$. Differentiating in $t$, taking traces and recalling that $\tr_g \dot{g} \equiv 0$, we get
\[ \dot{h} = 2 \, \varepsilon \, t \, g + (1 + \varepsilon \, t^2) \, \dot{g} \text{ and } \tr_h \dot{h} = 4 \, \varepsilon \, t \, (1 + \varepsilon \, t^2)^{-1} \text{.} \]
Differentiating and tracing again,
\[ \ddot{h} = 2 \, \varepsilon \, g + 4 \, \varepsilon \, t \, \dot{g} + (1 + \varepsilon \, t^2)^{-1} \, \ddot{g} \text{ and } \tr_h \ddot{h} = 4 \, \varepsilon \, (1 + \varepsilon \, t^2)^{-1} + \tr_g \ddot{g} \text{.} \]
The scalar curvature of the warped product metric $\gamma$ is
\begin{align*} 
R_\gamma & = 2 \, K_h - 2 \, v^{-1} \, \Delta_h v + v^{-2} \Big[ - \tr_h \ddot{h} - \frac{1}{4} \, (\tr_h \dot{h})^2 + \frac{\partial_t v}{v} \, \tr_h \dot{h} + \frac{3}{4} \, |\dot{h}|_h^2 \Big] \\
    & = 2 \, \mu + v^{-2} \Big[ - \tr_h \ddot{h} - \frac{1}{4} \, (\tr_h \dot{h})^2 + \frac{\partial_t v}{v} \, \tr_h \dot{h} + \frac{3}{4} \, |\dot{h}|_h^2 \Big] \\
    & \geq 2 \, \mu + v^{-2} \Big[ - \tr_h \ddot{h} + \frac{\partial_t v}{v} \, \tr_h \dot{h} \Big] \\
    & = 2 \, (1+\varepsilon \, t^2)^{-1} \, \lambda \\
    & \qquad + A^{-2} \, u^{-2} \Big[ - 4 \, \varepsilon \, (1+\varepsilon \, t^2)^{-1} - \tr_g \ddot{g} + 4 \, \varepsilon \, t \, (1+\varepsilon \, t^2)^{-1} \,  \frac{\partial_t u}{u} \Big] \\
    & = A^{-2} \, (1+\varepsilon \, t^2)^{-1} \, u^{-2} \Big[ 2 \, A^2 \, \lambda \, u^2 - 4 \, \varepsilon - (1+\varepsilon \, t^2) \, \tr_g \ddot{g} + 4 \, \varepsilon \, t \, \frac{\partial_t u}{u} \Big] \text{.}
\end{align*}
To get $R_\gamma > 0$ we simply need to arrange for
\[ 2 \, A^2 \, \lambda \, u^2 - 4 \, \varepsilon - (1+\varepsilon \, t^2) \, \tr_g \ddot{g} + 4 \, \varepsilon \, t \, \frac{\partial_t u}{u} > 0 \text{.} \]
Since $\lambda > 0$ and $u > 0$ we can estimate (term by term) that this quantity is
\[ \geq 2 \, A^2 \, \inf_t \lambda \, \inf_{t, x} u^2 - 4 \, \varepsilon - 2 \, \sup_t |\tr_g \ddot{g}| - 4 \, \varepsilon \, \sup_{t,x} | \partial_t \log u | \text{,} \]
which is indeed positive when $A \geq A_0 \gg 1$ and $0 < \varepsilon \leq \varepsilon_0 \ll 1$, with $A_0$, $\varepsilon_0$ depending on $\inf_t \lambda$, $\inf_{t,x} u$, $\sup_t |\tr_g \ddot{g}|$, and $\sup_{t,x} |\partial_t \log u|$.
\end{proof}

\begin{rema}
At this point we replace $u(t, x)$ by $A_0 \, u(t, x)$ and work with the family
\[ \gamma_\varepsilon = (1+\varepsilon \, t^2) \, g(t) + u(t, x)^2 \, dt^2 \]
of $\varepsilon \in (0,\varepsilon_0]$ dependent metrics on $\SS^2 \times [0,1]$. We will vary $\varepsilon$ later on, but only towards smaller values.
\end{rema}

%! Figure
\begin{figure}[ht!]
\label{figure}
\begin{tikzpicture}

  \begin{scope}[shift={(-1.5,0)}] % funny g(0) metric in M_+
    \draw [smooth, black!100] plot coordinates {(0, 1) (0.08, 0.98) (0.2, 0.8) (0.2, 0.6) (0.3, 0.4) (0.2, 0.2) (0.2, 0.0) (0.1, -0.2) (0.2, -0.4) (0.25, -0.6) (0.2, -0.8) (0.08, -0.98) (0, -1)};
    \draw [smooth, black!60] plot coordinates {(0, -1) (-0.08, -0.98) (-0.2, -0.8) (-0.2, -0.6) (-0.3, -0.4) (-0.2, -0.2) (-0.2, 0.0) (-0.1, 0.2) (-0.2, 0.4) (-0.25, 0.6) (-0.2, +0.8) (-0.08, +0.98) (0, 1)};
  \end{scope}
  \begin{scope}[shift={(1.5, 0)}] % round metric
    \begin{scope} % right half of round metric
      \clip (0, -1) rectangle (0.3, 1);
      \draw [black!100] (0, 0) ellipse (0.25 and 1);
    \end{scope}
    \begin{scope} % left half of round metric, dashed
      \clip (-0.3, -1) rectangle (0, 1);
      \draw [black!100, dotted] (0, 0) ellipse (0.25 and 1);
    \end{scope}
  \end{scope}
  \begin{scope}[shift={(1.2, 0)}] % round metric
    \begin{scope} % right half of round metric
      \clip (0, -1) rectangle (0.3, 1);
      \draw [black!100] (0, 0) ellipse (0.25 and 1);
    \end{scope}
    \begin{scope} % left half of round metric, dashed
      \clip (-0.3, -1) rectangle (0, 1);
      \draw [black!60, dotted] (0, 0) ellipse (0.25 and 1);
    \end{scope}
  \end{scope}
  \begin{scope}[shift={(0.9, 0)}] % round metric
    \begin{scope} % right half of round metric
      \clip (0, -1) rectangle (0.3, 1);
      \draw [black!100] (0, 0) ellipse (0.25 and 1);
    \end{scope}
    \begin{scope} % left half of round metric, dashed
      \clip (-0.3, -1) rectangle (0, 1);
      \draw [black!60, dotted] (0, 0) ellipse (0.25 and 1);
    \end{scope}
  \end{scope}
  \begin{scope}[shift={(0.6, 0)}] % round metric
    \begin{scope} % right half of round metric
      \clip (0, -1) rectangle (0.3, 1);
      \draw [black!100] (0, 0) ellipse (0.25 and 1);
    \end{scope}
    \begin{scope} % left half of round metric, dashed
      \clip (-0.3, -1) rectangle (0, 1);
      \draw [black!60, dotted] (0, 0) ellipse (0.25 and 1);
    \end{scope}
  \end{scope}

  \begin{scope} % cylinder
    \draw [black!100] plot coordinates {(-1.5, 1) (1.5, 1)};
    \draw [black!100] plot coordinates {(-1.5, -1) (1.5, -1)};
  \end{scope}

  \node at (-1.5, -2) {$g(0)$}; \draw [black!100, ->] plot coordinates {(-1.5, -1.6) (-1.5, -1.1)}; % g(0)
  \node at (0.6, -2) {$g(t)$}; \draw [black!100, ->] plot coordinates {(0.6, -1.6) (0.6, -1.1)}; % g(t)
  \node at (1.5, -2) {$g(1)$}; \draw [black!100, ->] plot coordinates {(1.5, -1.6) (1.5, -1.1)}; % g(1)

  \begin{scope} % cylinder with angle
    \draw [black!100, dashed] plot coordinates {(-1.5, 1) (-1.3, 1.01) (-0.8, 1.02) (-0.2, 1.07) (0.4, 1.12) (0.9, 1.2) (1.5, 1.4)};
    \draw [black!100, dashed] plot coordinates {(-1.5, -1) (-1.3, -1.01) (-0.8, -1.02) (-0.2, -1.07) (0.4, -1.12) (0.9, -1.2) (1.5, -1.4)};
    \draw [black!100, dashed] (1.5, 0) ellipse (0.5 and 1.4);
    \node at (0.7, 2.3) {$(1+\varepsilon \, t^2) \, g(t)$}; \draw [black!100, ->] plot coordinates {(0.9, 1.9) (0.9, 1.3)}; % (1+et^2) g(t)
  \end{scope}

  \begin{scope}[shift={(4,0)}] % Schwarzschild
    \begin{scope} % right half of round metric
      \clip (0, -1) rectangle (0.3, 1);
      \draw [black!100] (0, 0) ellipse (0.25 and 1);
    \end{scope}
    \begin{scope} % left half of round metric, dashed
      \clip (-0.3, -1) rectangle (0, 1);
      \draw [black!60] (0, 0) ellipse (0.25 and 1);
    \end{scope}
    \begin{scope} % neck
      \draw [smooth, black!100] plot coordinates {(0, 1) (0.2, 1.01) (0.3, 1.05) (0.5, 1.2) (0.7, 1.5) (0.8, 1.8)};
      \draw [smooth, black!100] plot coordinates {(0, -1) (0.2, -1.01) (0.3, -1.05) (0.5, -1.2) (0.7, -1.5) (0.8, -1.8)};
    \end{scope}
  \end{scope}
    
  \begin{scope}[shift={(4,0)}] % dashed Shwarzschild
    \draw [black!100, dashed] (0, 0) ellipse (0.55 and 1.5); % horizon
    \begin{scope} % neck
      \draw [smooth, black!100, dashed] plot coordinates {(0, 1.5) (0.2, 1.51) (0.3, 1.57) (0.5, 1.8) (0.7, 2.25) (0.8, 2.7)};
      \draw [smooth, black!100, dashed] plot coordinates {(0, -1.5) (0.2, -1.51) (0.3, -1.57) (0.5, -1.8) (0.7, -2.25) (0.8, -2.7)};
    \end{scope}
    \node at (4.2, 0) {Schwarzschild $m > \sqrt{\frac{\area(g)}{16 \, \pi}}$};
    \draw [black!100, ->] plot coordinates {(1.8, 0) (0.7, 0)};
  \end{scope}
        
\end{tikzpicture}
\caption{Any $g(0) \in \sM_+$ can be joined to a round metric $g(1)$ via a path $g(t) \in \sM_+$ with pointwise constant area form (Lemma \ref{extension.lemma}). The metric $(1+\varepsilon \, t^2) \, g(t) + dt^2$ is foliated by mean convex spheres ($t =$ const) and can be warped to a metric $\gamma_\varepsilon$ of positive scalar curvature for $\varepsilon$ sufficiently small (Lemma \ref{collar.lemma}). The next section deals with gluing exterior Schwarzschild regions $m > \sqrt{\area(g)/16 \pi}$ to $\gamma_\varepsilon$ as $\varepsilon \downarrow 0$ and $m \downarrow \sqrt{\area(g)/16\pi}$.}
\end{figure}
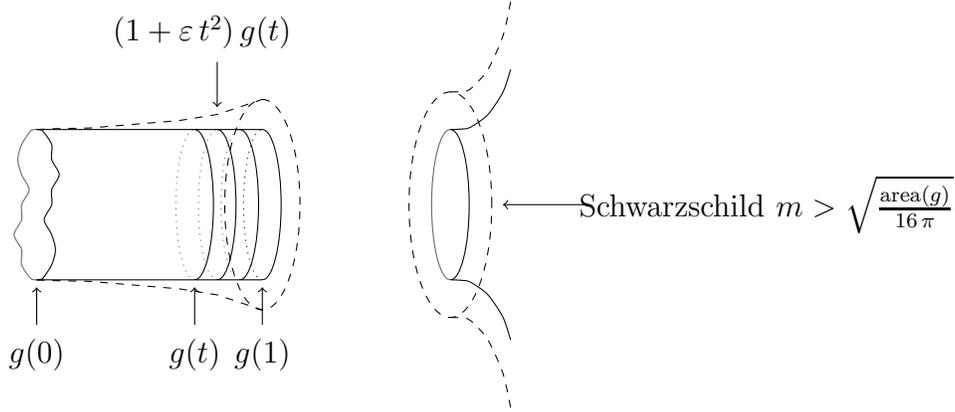

\section{Smoothly gluing exterior Schwarzschild regions}

\label{sec:gluing.schwarzschild}

The following extension result is the main theorem of this section:

\begin{theo}
\label{main.theorem}
Let $g \in \sM_+$. For any $m > 0$ such that $16 \, \pi \, m^2 > \area(g)$ there exists an asymptotically flat 3-manifold $M^3$ with non-negative scalar curvature such that
\begin{enumerate}[(i)]
\item $\partial M^3$ is isometric to $(\SS^2, g)$ and minimal,
\item $M^3$ is isometric to a mass-$m$ Schwarzschild metric outside a compact set, and
\item $M^3$ is foliated by mean convex spheres that eventually coincide with standard Schwarzschild coordinate spheres.
\end{enumerate}
\end{theo}

\begin{rema}
By the maximum principle applied to the mean convex spherical foliation, $M^3$ will contain no interior closed minimal surfaces. In view of the Riemannian Penrose inequality $16 \, \pi \, m_{\operatorname{ADM}}(M^3)^2 \geq \area(g)$ this theorem lets us construct an almost minimizing sequence of smooth extensions of the Bartnik data $(\SS^2, g, H = 0)$. We conclude that $m_B(\SS^2, g, H = 0) = (\area(g)/16\pi)^{1/2}$ for all $g \in \sM_+$.
\end{rema}

\begin{rema}
We can also arrange for $\partial M^3$ to be totally geodesic as in the remark following the collar lemma, Lemma \ref{collar.lemma}.
\end{rema}

We need to first study the positive scalar curvature condition on topological cylinders $\SS^2 \times [0,T]$ endowed with metrics of the particular form $f(t)^2 \, g_* + dt^2$, where $g_*$ is a round metric on $\SS^2$ with area $4 \, \pi$. The function $f$ is required to be positive and smooth.

We introduce a piece of notation that will be used in this section. For $\alpha \in (0, \infty)$, $\beta \in \RR$, we write
\[ \Omega(\alpha, \beta) = (-\infty, \frac{1}{2} \, \alpha \, (1 - \alpha^{-2} \beta^2)) \text{,} \]
as shorthand for the corresponding open half-infinite interval in $\RR$. We will be using two basic facts about $\Omega(\alpha, \beta)$.

\begin{fact}
$\Omega(\alpha, \beta)$ shrinks as a set when $\alpha$ shrinks and $\beta$ is held fixed, or when $\beta$ grows and $\alpha$ is held fixed. In particular, if $P \in \RR_+ \times \RR$ is above and to the left of $P' \in \RR_+ \times \RR$, then $\Omega(P) \subseteq \Omega(P')$.
\end{fact}

\begin{fact}
The metric $f(t)^2 \, g_* + dt^2$ on $\SS^2 \times [0,T]$ has positive scalar curvature if and only if $f''(t) \in \Omega(f(t), f'(t))$ for all $t \in [0,T]$.
\end{fact}

The latter is a consequence of the formula for the scalar curvature of warped product metrics. Now follows a technical lemma that will be key in our proof of the main theorem.

\begin{lemm}[Gluing lemma]
\label{gluing.lemma}
Let $f_i : [a_i, b_i] \to \RR$, $i = 1$, $2$ be two smooth functions with
\begin{enumerate}[(I)]
\item $f_i$, $f_i'$, $f_i'' > 0$, $i = 1$, $2$,
\item $f_i''(t) \in \Omega(f_i(t), f_i'(t))$ for all $t \in [a_i, b_i]$, $i = 1$, $2$, and
\item which satisfy the compatibility conditions $f_1(b_1) < f_2(a_2)$ and $f_1'(b_1) = f_2'(a_2)$.
\end{enumerate}
Then, after translating the intervals so that $a_2 - b_1 = (f_2(a_2) - f_1(b_1))/f_1'(b_1)$, we can construct a smooth $f : [a_1, b_2] \to \RR$ such that:
\begin{enumerate}[(i)]
\item $f$, $f' > 0$ on $[a_1, b_2]$,
\item $f \equiv f_1$ on $[a_1, \frac{a_1+b_1}{2}]$,
\item $f \equiv f_2$ on $[\frac{a_2+b_2}{2}, b_2]$, and
\item $f(t)^2 \, g_* + dt^2$ is scalar positive on $\SS^2 \times [a_1, b_2]$.
\end{enumerate}
\end{lemm}
\begin{proof}
Define $\widetilde{f} : [a_1, b_2] \to \RR$ to extend $f_1$ and $f_2$ by interpolating linearly between them. In view of the provided hypotheses, $\widetilde{f}$ is strictly increasing and $C^{1,1}[a_1,b_2]$. For $\sigma > 0$ let $f_\sigma : [a_1, b_2] \to \RR$ be obtained from $\widetilde{f}$ by mollification, leaving the function fixed on $[a_1, \frac{a_1+b_1}{2}]$ and $[\frac{a_2+b_2}{2}, b_2]$. We claim that $f_\sigma$, for $\sigma > 0$ small enough, can play the role of the postulated function $f$.

Write $\Omega[\widetilde{f}]$ for the Lipschitz-continuous graph of $[a_1, b_2] \ni t \mapsto \sup \Omega(\widetilde{f}(t), \widetilde{f}'(t))$; define $\Omega[f_\sigma]$ analogously. Write $T_i$, $i = 1, 2$, for the graph $f_i''$ over $[a_i, b_i]$ and $T_\sigma$ for that of $f_\sigma''$ over $[a_1, b_2]$.

In view of (II), $T_1 \cup T_2$ lies strictly below $\Omega[\widetilde{f}]$, and by (I) all lie above the $t$-axis. Write $d > 0$ for the minimal vertical distance between $T_1 \cup T_2$ and $\Omega[\widetilde{f}]$. Since $\widetilde{f} \in C^{1,1}$, we know that $f_\sigma \to f$ in $C^1$ and thus $\Omega[f_\sigma] \to \Omega[\widetilde{f}]$ uniformly, so for small enough $\sigma$ we can arrange for $\Omega[f_\sigma]$ to be within $d/2$ of $\Omega[\widetilde{f}]$. Again, since $\widetilde{f} \in C^{1,1}[a_1,b_2]$ and $C^2$ away from $\{a_2, b_1\}$, $\widetilde{f}'' > 0$ on $[a_1, b_1) \cup (a_2, b_2]$ and $\widetilde{f}'' = 0$ on $(a_2, b_1)$, for small enough $\sigma$ we know by properties of compactly supported mollifications that $T_\sigma$ lies within $d/2$ of $T_1 \cup T_2 \cup ([a_2,b_1] \times \{0\})$, and is therefore strictly below $\Omega[f_\sigma]$. The conclusion follows.
\end{proof}

Before we present the proof of the main theorem, we need to make a few remarks on Riemannian Schwarzschild metrics. We remind the reader that the Riemannian mass-$m$ Schwarzschild is a scalar-flat metric on $\SS^2 \times (2m, \infty)$ given by
\[ g_{S,m} = t^2 \, g_* + \Big( 1 - \frac{2m}{t} \Big)^{-1} \, dt^2 \text{,} \]
with $g_*$ a round metric on $\SS^2$ with area $4 \, \pi$. This is the exterior part of the static metric. It will be more convenient for us to view this in the form of the warped product metrics as needed in the gluing lemma; i.e.,
\[ g_{S,m} = u_m(s)^2 \, g_* + ds^2 \]
on $\SS^2 \times [0,\infty)$, where $s$ is the radial geodesic arclength, and $s = 0$ corresponds to the horizon. The function $u_m$ is manifestly seen to satisfy
\begin{equation}
\label{eq:schwarzschild.warped.prod.form}
\tag{$\dagger$}
u_m(0) = 2 \, m, \quad u_m'(0) = 0, \quad u_m'(s) = \Big( 1 - \frac{2m}{u_m(s)} \Big)^{1/2} \quad \text{and} \quad u_m''(s) = \frac{m}{u_m(s)^2} \quad \text{for } s > 0 \text{.}
\end{equation}

The end goal is to glue $g_{S,m}$ near some region $\{ s = s_0 \}$, $0 < s_0 \ll 1$, to the scalar positive collar $\SS^2 \times [0,1]$ constructed in the previous section. In order to do this we will first need to "bend" $g_{S,m}$ near $\{ s = s_0 \}$ and make it scalar positive on an annular region of its own. This positivity will afford us certain freedom in gluing.

\begin{lemm}[Bending lemma]
\label{bending.lemma}
Let $s_0 \in (0, \infty)$ be fixed. For $0 < \delta \ll 1$ there exists a smooth function $\sigma : [s_0 - \delta, \infty) \to (0,\infty)$ such that:
\begin{enumerate}[(i)]
\item $\sigma(s) = s$ for all $s \geq s_0$,
\item $\sigma$ is monotonically increasing, and
\item the metric $u_m(\sigma(s))^2 \, g_* + ds^2$ is scalar flat when $s \geq s_0$ and scalar positive when $s_0 - \delta \leq s < s_0$.
\end{enumerate}
\end{lemm}
\begin{proof}
Recalling $R_{f(s)^2 \, g_* + ds^2} = 2 \, f^{-2} \big[ 1 - \, f^{-2} \, (f')^2  - 2 \, f^{-1} \, f'' \big]$, what we need to ensure is that for $s < s_0$,
\begin{align*} 
& 1 - u_m(\sigma(s))^{-2} \Big[ \frac{d}{ds} u_m(\sigma(s)) \Big]^2 - 2 \, u_m(\sigma(s))^{-1} \frac{d^2}{ds} u_m(\sigma(s)) > 0 \\
& \qquad \Leftrightarrow 1 - u_m^{-2} \, (u_m')^2 \, (\sigma')^2 - 2 \, u_m^{-1} \, u_m'' \, (\sigma')^2 - 2 \, u_m^{-1} \, u_m' \sigma'' > 0 \\
& \qquad \Leftrightarrow 1 - (\sigma')^2  + (\sigma')^2 \big[ 1 - u_m^{-2} \, (u_m')^2 - 2 \, u_m^{-1} \, u_m'' \big] - 2 \, u_m^{-1} \, u_m' \, \sigma'' > 0 \text{.} \\
\end{align*}
The term inside the bracket vanishes in view of Schwarzschild being scalar flat, so the inequality collapses to $1 - (\sigma')^2 - 2 \, u_m^{-1} \, u_m' \, \sigma'' > 0$. It will suffice to construct a smooth $\theta : [s_0 - \delta, s_0] \to \RR$ to play the role of $\sigma'$, provided it satisfies
\begin{enumerate}[(i')]
\item $\theta(s_0) = 1$, $\theta'(s_0) = \theta''(s_0) = \ldots = 0$,
\item $\theta > 0$ on $[s_0 - \delta, s_0]$, and
\item $1 - \theta^2 - 2 \, u_m^{-1} \, u_m' \, \, \theta' > 0$ on $[s_0 - \delta, s_0)$;
\end{enumerate}
having done so we integrate to recover $\sigma$ to the left of $s_0$, which can be smoothly extended to $\sigma(s) \equiv s$ to the right of $s_0$. One may check directly that $\theta(s) = 1 + e^{-1/(s-s_0)^2}$ for $s < s_0$, $\theta(s_0) = 1$, works.
\end{proof}

\begin{proof}[Proof of Theorem \ref{main.theorem}]
Start by constructing a scalar positive collar extension $(\SS^2 \times [0,1], \gamma_\varepsilon)$ of $g$ as in the previous section, with $\varepsilon \leq \varepsilon_0$ arbitrary for the time being. Let $g_*$ be a round metric on $\SS^2$ with area $4 \, \pi$ and such that $g(t) \equiv \rho^2 \, g_*$ for $t \in [1/2,1]$ on the collar extension. Since $u(t, \cdot) \equiv T > 0$ for $t \in [1/2, 1]$, changing variables to $s = T \, t$ makes the $t \in [1/2, 1]$ deformation region of the collar $\SS^2 \times [0,1]$ look like
\[ \gamma_\varepsilon = (1 + \varepsilon \, T^{-2} \, s^2) \, \rho^2 \, g_* + ds^2 \text{,} \]
with $s \in [T/2, T]$. Write $f_\varepsilon(s) = (1 + \varepsilon \, T^{-2} \, s^2)^{1/2} \, \rho$. Clearly
\begin{equation} 
\label{eq:collar.region.warped.prod.form}
\tag{$\dagger'$}
f_\varepsilon'(s) = (1 + \varepsilon \, T^{-2} \, s^2)^{-1/2} \, \varepsilon \, T^{-2} \, \rho \, s \text{ and } f_\varepsilon''(s) = (1+ \varepsilon \, T^{-2} \, s^2)^{-3/2} \, \varepsilon \, T^{-2} \, \rho
\end{equation}
hold true. Furthermore
\begin{equation} 
\label{eq:collar.region.curve.slope}
\tag{$\ddagger$}
\frac{f_\varepsilon'(T)}{f_\varepsilon(T)} = \frac{\varepsilon \, T^{-1}}{1 + \varepsilon} = \frac{1}{T} \, \Big( 1 - \frac{1}{1+ \varepsilon} \Big) \text{,}
\end{equation}
i.e., the slope traced out by the curve $\varepsilon \mapsto \Gamma(\varepsilon) \triangleq (f_\varepsilon(T), f_\varepsilon'(T))$ that lies on the first quadrant is decreasing as $\varepsilon \downarrow 0$. The curve converges as $\varepsilon \downarrow 0$ to the point
\[ \lim_{\varepsilon \downarrow 0} \Gamma(\varepsilon) = (\rho, 0) = \Big( \sqrt{\frac{\area(g(0))}{4 \, \pi}}, 0 \Big) \text{.} \]

Let $m$ be as in the statement of the theorem. Notice that the curve $s \mapsto \Delta(s) \triangleq (u_m(s), u_m'(s))$ that also lies on the first quadrant converges to $(2 \, m, 0)$ as $s \downarrow 0$. This point lies to the right of the curve $\Gamma$ by our choice of $m$, so for small enough $s_0 > 0$ curve $[0,s_0] \ni s \mapsto \Delta(s)$ lies strictly below and to the right of the segment $[0, \varepsilon_0] \ni \varepsilon \mapsto \Gamma(\varepsilon)$.

Apply the bending lemma (Lemma \ref{bending.lemma}), modify the curve $\Delta(s)$ on a small time interval $[s_0 - \delta, s_0]$ and call the new curve $[s_0 - \delta, s_0] \ni s \mapsto \widetilde{\Delta}(s) = (u_m(\sigma(s)), (u_m(\sigma(s)))')$. For small enough $\delta > 0$ the curve $\widetilde{\Delta}$ will continue to be below and to the right of $\Gamma$.

Finally we claim that we may apply the gluing lemma (Lemma \ref{gluing.lemma}) to form a scalar positive bridge between the collar region and the bent Schwarzschild exterior region.

The role of $f_2$ in the gluing lemma will be played by $[s_0 - \delta, s_0] \ni s \mapsto u_m(\sigma(s))$. Condition (I) of the lemma is satisfied by virtue of \eqref{eq:schwarzschild.warped.prod.form} after possibly shrinking $\delta > 0$ some more to retain strict convexity. Condition (II) is guaranteed to hold by the bending lemma, which bends a thin annular region of Schwarzschild to positive scalar curvature.

The role of $f_1$ will be played by $[T/2, T] \ni s \mapsto f_\varepsilon(s)$, for an $\varepsilon \leq \varepsilon_0$ that is to be determined. Condition (I) of the lemma is satisfied by virtue of \eqref{eq:collar.region.warped.prod.form}. Condition (II) is guaranteed to hold since the collar is scalar positive.

Having already chosen $\delta$, we choose $\varepsilon \leq \varepsilon_0$ to be such that $\Gamma(\varepsilon)$ and $\widetilde{\Delta}(s_0-\delta)$ align horizontally. There is a unique such choice in view of \eqref{eq:collar.region.curve.slope}. This fulfills the matching slope half of condition (III) in the lemma. The remaining half of (III) is fulfilled automatically since $\Gamma$ lies to the 
left of $\widetilde{\Delta}$.

The gluing lemma then creates a scalar positive bridge $f(s)^2 \, g_* + ds^2$ that preserves the $t \leq 1/2$ collar region and the $s \geq s_0$ exterior Schwarzschild region. It also guarantees a mean convex foliation since $f' > 0$. The result follows.
\end{proof}

\section{New examples of horizons}

\label{sec:horizon.examples}

Sections \ref{sec:collar.extensions} and \ref{sec:gluing.schwarzschild} provide a complete characterization of non-degenerate apparent horizons, showing they are in bijective correspondence with $\sM_+$. While analytically complete, the $\sM_+$ characterization has the disadvantage of not being geometrically transparent. The examples of horizons discussed in the literature usually have positive curvature (see, e.g., \cite{lin-sormani-ricci-flow}, \cite{bartnik-quasispherical-metrics}), which seems in no way necessary given the operator. In this section we show that horizons with arbitrarily large negative integral curvature exist and in fact form a dense subset among all non-degenerate horizons in the appropriate topology.

\begin{theo}
\label{negative.curv.eigenvalue.metrics}
The set $\sM_+$ is relatively open in every pointwise conformal class of metrics, with respect to the uniform $C^1$ topology. Furthermore, for every $c > 0$ the subset of $\sM_+$ consisting of metrics $g$ with $\int_{\SS^2} (K_g)_- \, dA_g \geq c$ is $C^1$ dense in $\sM_+$ where $(K_g)_-=\max\{0,-K_g\}$.
\end{theo}

\begin{rema}
Once we have established that a metric is in $\sM_+$ we can extend it to an initial data set using Theorem \ref{main.theorem}. Alternatively, we can perform the extension in \cite{smith-black-holes-prescribed-full}.
\end{rema}

This theorem requires that we gain further insight on $\sM_+$ or, equivalently, the operator $L_g = -\Delta_g + K_g$. For $w \in C^\infty(\SS^2)$ and a round metric $g_*$ of area $4 \, \pi$ we will be interested in the quasilinear elliptic operator
\[ \cQ_w \phi = -\Delta_* \phi - |\nabla^* (\phi - w)|_{*}^2 + 1 \text{.} \]
The following proposition relates $\sM_+$ with $\cQ_w$.

\begin{prop}
\label{c1.characterization.mplus}
The metric $g = e^{2w} \, g_*$ is in $\sM_+$ if and only if $\cQ_w \phi > 0$ on $\SS^2$ for some $\phi \in C^\infty$.
\end{prop}

\begin{rema}
Choosing $\phi \equiv w$ we recover the known fact that every metric $g$ with positive curvature is in $\sM_+$. The proposition is a significant strengthening of this.
\end{rema}

\begin{proof}[Proof of Proposition \ref{c1.characterization.mplus}]
We first claim that
\[ g \in \sM_+ \text{ if and only if there exists a } f \in C^\infty, \; f > 0, \text{ with } L_g f > 0 \text{.} \]

The forward direction is clear: pick $f$ to be the first eigenfunction of $L_g$. The backward direction goes by contradiction. Suppose $f > 0$ is as described but that $g \not \in \sM_+$. Then $\lambda_1(L_g) \leq 0$ so the first eigenfunction $h > 0$ of $L_g$ satisfies $L_g h = \lambda \, h$ with $\lambda \leq 0$. Since $\SS^2$ is compact and neither one of $f$, $h$ vanishes, there must exist a constant $c > 0$ and a point $x_0 \in \SS^2$ such that $f - c \, h \geq 0$ and $f(x_0) = c \, h(x_0)$. The map $f - c \, h$ attains a global minimum of zero at $x_0$, so $L_g(f - c \, h)(x_0) \leq 0 \Leftrightarrow L_g f(x_0) \leq c \, L_g h(x_0)$, which is impossible because the left hand side is positive and the right hand side is non-positive. This confirms the claim.

Since every smooth $f > 0$ can be written as $e^\zeta$, $L_g f = L_g e^\zeta = e^\zeta \, \big( -\Delta_g \zeta - |\nabla^g \zeta|_g^2 + K_g \big)$, so the previous claim can be restated as:
\[ g \in \sM_+ \text{ if and only if there exists } \zeta \in C^\infty \text{ with } -\Delta_g \zeta - |\nabla^g \zeta|_g^2 + K_g > 0 \text{.} \]

Finally, writing $g = e^{2w} \, g_*$ for a round metric $g_*$ of area $4 \, \pi$ we see that
\begin{align*}
& -\Delta_g \zeta - |\nabla^g \zeta|_g^2 + K_g \\
& \qquad = e^{-2w} \, \big( - \Delta_* \zeta - |\nabla^* \zeta|_{*}^2 + 1 - \Delta_* w \big) \\
& \qquad = e^{-2w} \, \big( -\Delta_* (\zeta + w) - |\nabla^* \zeta|_{*}^2 + 1 \big) \text{,}
\end{align*}
and the conclusion follows with the particular choice $\phi = \zeta + w$.
\end{proof}

\begin{proof}[Proof of Theorem \ref{negative.curv.eigenvalue.metrics}]
Fix $v \in C^\infty$ with $e^{2v} \, g_* \in \sM_+$. By Proposition \ref{c1.characterization.mplus} there must exist $\phi \in C^\infty$ with $\cQ_v \phi > 0$. For any $w \in C^\infty$ and the same unit round metric $g_*$,
\begin{align*}
\cQ_w \phi - \cQ_v \phi & = \big( - \Delta_* \phi - |\nabla^* (\phi-w)|_{*}^2 + 1 \big) - \big( - \Delta_* \phi - |\nabla^* (\phi-v)|_{*}^2 + 1 \big) \\
    & = |\nabla^* (\phi-v)|_{*}^2 - |\nabla^* (\phi-w)|_{*}^2 \\
    & = g_*(\nabla^* (2\phi - w - v), \nabla^* (w-v)) \\
    & \geq - |\nabla^* (2\phi-w-v)|_{*} \, |\nabla^* (w-v)|_{*} \text{,}
\end{align*}
so $\cQ_w \phi \geq \cQ_v \phi - |\nabla^* (2\phi-w-v)|_{*} \, |\nabla^* (w-v)|_{*}$ is positive provided $|\nabla^* (w-v)|_{*}$ is small enough.

Next we need to establish the density claim. Let $v$, $\phi$ be as before, and $n \in \{1, 2, \ldots\}$, $\alpha \in (0,1)$ be parameters that are to be determined. Without loss of generality write $g_* = d\theta^2 + \sin^2 \theta \, d\phi^2$. Define $h = - n^{-1} \, \cos(n\theta)$, which satisfies $\nabla^* h = \sin (n\theta) \, \frac{\partial}{\partial \theta}$ and $\Delta_* h = n \, \cos(n\theta) + \sin (n\theta) \, \cot \theta$ and extends smoothly across the north ($\theta = 0$) and south ($\theta = \pi$) poles. We now choose $\alpha \in (0,1)$ to be small enough that $e^{2(v + \alpha h)} \, g_* \in \sM_+$, by $C^1$ openness. This can be done independently of $n$. The metric $g = e^{2(v + \alpha h)} \, g_*$ satisfies
\[ \int_{\SS^2} (K_g)_- \, dA_g = \int_{\SS^2} \big( 1 - \Delta_* v - \alpha \, \Delta_* h \big)_- \, dA_* = \int_{\SS^2} \big( 1 - \Delta_* v - \alpha \, n \, \cos(n\theta) - \alpha \, \sin (n\theta) \, \cot \theta \big)_- \, dA_* \text{.} \]
If $D = \{ \pi/3 \leq \theta \leq 2\pi/3 \}$ then $1 - \Delta_* v - \alpha \, \sin (n\theta) \, \cot(\theta) \leq \Lambda = \Lambda(v, \alpha)$ on $D$ and $\cos(n\theta) \geq 1/2$ on a set $D_n \subset D$ with $\area(D_n, g_*) \geq \mu > 0$ independently of $n$. Altogether this gives
\[ \int_{\SS^2} (K_g)_- \, dA_g \geq \int_{D_n} \big( \Lambda - \frac{1}{2} \, \alpha \, n \big)_- \, dA_* \geq \mu \, \big( \Lambda - \frac{1}{2} \, \alpha \, n \big)_- \text{,} \]
which can be made arbitrarily large by sending the still free parameter $n \uparrow \infty$. The density claim follows, since this can be done for all small enough $\alpha \in (0,1)$.
\end{proof}

There is an immediate corollary to the computation performed to check for relative openness above.

\begin{coro}
If $g_*$ is a round metric with area $4 \, \pi$ and $w \in C^\infty$ is such that $|\nabla^* w|_{*} < 1$, then $e^{2w} \, g_* \in \sM_+$.
\end{coro}
\begin{proof}
Choose $v = \phi = 0$ in the computation carried out in the proof of the theorem. It follows that $\cQ_w 0 \geq \cQ_0 0 - |\nabla^* w|_{*}^2 = 1 - |\nabla^* w|_{*}^2$, which is clearly positive when $|\nabla^* w|_* < 1$.
\end{proof}

\section{Gibbons' formulation of Thorne's hoop conjecture}

\label{sec:gibbons.hoop.conjecture}

This last section is relevant to Thorne's hoop conjecture in general relativity, \cite{thorne-hoop-conjecture}, \cite{cvetic-gibbons-pope-hoop-conjecture}: \emph{"Horizons form if and only if a mass $m$ gets compacted into a region whose circumference in every direction is $\leq 4 \, \pi \, m$."} This conjecture has been invoked in numerical relativity and in studies of hole scattering in four and higher dimensions, and it has been suggested that it may provide a route to a precise formulation of the idea that there is a minimal length in quantum gravity \cite{cvetic-gibbons-pope-hoop-conjecture}.

The notion of "circumference" in this conjecture is up for interpretation. Reformulating this conjecture in a mathematically precise manner is an issue of mathematical and physical significance. Gibbons posed the following mathematically precise formulation of the conjecture in \cite{gibbons-hoop-conjecture} with the Birkhoff invariant of the horizon playing the role of circumference:

($G$) \emph{"For an outermost marginally trapped surface $(\SS^2, g)$ lying in a Cauchy hypersurface with ADM mass $m_{\operatorname{ADM}}$ on which the Dominant Energy Condition holds, $\beta(g) \leq 4 \, \pi \, m_{\operatorname{ADM}}$."}

The Birkhoff invariant $\beta(g)$ of a sphere $(\SS^2, g)$ is defined to be
\[ \beta(g) = \inf_f \sup_{c \in \RR} \length(f^{-1}(c), g) \text{,} \]
the infimum being taken over all functions $f: \SS^2 \to \RR$ with only two critical points. Intuitively, $\beta(g)$ represents the least length required of a closed elastic rubber band to allow it to be slid over the sphere.

In addition to giving a mathematically precise formulation of the conjecture, \cite{gibbons-hoop-conjecture} shows that ($G$) holds true on general charged rotating Kerr-Newman black holes. More generally ($G$) has been checked to hold true on: the horizons of all four-charged rotating black hole solutions of ungauged supergravity theories, allowing for the presense of a negative cosmological constant; for multi-charged rotating solutions in gauged supergravity; and for the Ernst-Wild static black holes immersed in a magnetic field which are asymptotic to the Melvin solution  \cite{cvetic-gibbons-pope-hoop-conjecture}.

\cite{gibbons-hoop-conjecture} also posed a weaker form of the conjecture that uses the length $\ell(g)$ of the shortest nontrivial geodesic on $(\SS^2, g)$ in place of $\beta(g)$:

($G'$) \emph{"For an outermost marginally trapped surface $(\SS^2, g)$ lying in a Cauchy hypersurface with ADM mass $m_{\operatorname{ADM}}$ on which the Dominant Energy Condition holds, $\ell(g) \leq 4 \, \pi \, m_{\operatorname{ADM}}$."}

This conjecture is strictly weaker since $\beta(g)$ is known to be attained on a geodesic \cite{birkhoff-dynamical-systems}; i.e., $\ell(g) \leq \beta(g)$. It was shown in \cite{gibbons-hoop-conjecture} that ($G'$) always holds on horizons with an antipodal $\ZZ_2$ symmetry, by virtue of Pu's systolic inequality \cite{pu-systolic-inequalities}. Therefore there can be no counterexamples to ($G$) by way of disproving ($G'$) in the presense of such symmetries.

We present a counterexample to ($G'$) and thus also ($G$) that is motivated by a construction with a $\ZZ_3$ symmetry that is originally due to Croke \cite{croke-length-shortest-geodesic} and relies on our optimal Bartnik mass computation in Section \ref{sec:gluing.schwarzschild}.

We first describe a metric on the two sphere with larger Birkhoff invariant than that of a
sphere of the same area. The surface $\Sigma$ is gotten 
by considering the regular hexagon with vertices $2e^\frac{k\pi i}{3}$ for $k=0,1,\ldots, 5$.
It is a fundamental domain for the torus $\Sigma_0=\mathbb R^2/\Lambda$ where $\Lambda$ is the
lattice generated by $\omega_1=2\sqrt{3}e^{-\frac{\pi i}{6}}$ and $\omega_2=2\sqrt{3}e^\frac{\pi i}{6}$.
Now the rotation $R(z)=e^\frac{2\pi i}{3}z$ is well defined on the torus and fixes the 
equivalence classes modulo $\Lambda$ of the three points $0$, $2$, and $2e^\frac{\pi i}{3}$.
We denote the group generated by $R$ as $\mathbb Z_3$ and we let 
$\Sigma=\Sigma_0/\mathbb Z_3$. From the construction we can see that $\Sigma$ may be
thought of as two equilateral triangles of side length $2$ adjacent along a common edge
with the two sides from each of the common vertices identified in the direction from the vertex.
We see that $\Sigma$ is the two sphere with a metric which is flat away from three points
each of which is a cone point with cone angle $\frac{2\pi}{3}$. We have the following result
whose proof we include for completeness of the exposition. Note that the area of $\Sigma$
is equal to $2\sqrt{3}$.

\begin{lemm}
\label{triangle.geodesic.lemma}
Given any $\delta>0$ there exists a smooth approximation $\Sigma_\delta$ to $\Sigma$ with
non-negative curvature (in particular the metric on $\Sigma_\delta$ is in $\sM_+$) and 
such that the length of the shortest nontrivial closed geodesic of 
$\Sigma_\delta$ satisfies  $\ell(\Sigma_\delta) \geq 2 \sqrt{3}-\delta$ while the area satisfies 
$\area(\Sigma_\delta)\leq 2\sqrt{3}+\delta$.
\end{lemm}
\begin{proof}
We first observe that a cone with cone angle $\frac{2\pi}{3}$ can be approximated by a 
smooth metric with non-negative curvature which agrees with the cone metric outside
any chosen ball centered at the vertex. Furthermore we can take the approximating metric
to be rotationally symmetric about a point and such that the square of the distance function
from that point is a strictly convex function. This can be seen easily by isometrically embedding the cone into $\mathbb R^3$ as the graph of the function $z=\alpha\sqrt{x^2+y^2}$ with 
$\alpha=2\sqrt{2}$ chosen so that the cone angle is correct. We can then smooth the function 
keeping it convex and rotationally symmetric. By the Gauss-Bonnet theorem the total curvature of 
a ball whose boundary is in the flat region is $\frac{4\pi}{3}$, and since smaller balls have smaller
total curvature, the Gauss-Bonnet theorem implies that the geodesic curvature of the boundary
of any rotationally symmetric ball is a positive constant. This implies that the square of the 
distance function from the center of symmetry is strictly convex. 

Using the construction above we may approximate the metric defining $\Sigma$ by a metric 
$\hat{g}$ which is smooth, has non-negative curvature, 
and agrees with the metric of $\Sigma$ outside a small neighborhood of the cone points. We
show that if the neighborhood is chosen small enough we have $\area(\hat{g})\leq 2\sqrt{3}+\delta$
and $\ell(\hat{g})\geq 2\sqrt{3}-\delta$. We will then let $\Sigma_\delta=(\mathbb S^2,\hat{g})$. The first of these inequalities is clear since the area of $\Sigma$ is $2\sqrt{3}$. 

We now show that the length of any closed geodesic of $\hat{g}$ is at least $2\sqrt{3}-\delta$
if the neighborhoods are chosen small enough. First we observe that there can be no closed
geodesic lying entirely in the non-flat region near a cone point. This follows from the geodesic
convexity of the distance function and the maximum principle. Let $\gamma$ be a smooth closed
curve which intersects the flat portion of $\hat{g}$ and is a geodesic there. We claim that
$\length(\gamma)\geq 2\sqrt{3}-\delta$ if the neighborhood is chosen small enough. To see
this we consider three cases; first, suppose that $\gamma$ lies entirely in the flat region. In this
case we lift $\gamma$ beginning at some initial point $P$ to the flat torus torus $\Sigma_0$.
The lifted curve $\tilde{\gamma}$ is then a geodesic extending from the chosen lift $\tilde{P}$
to a point which is the image of $\tilde{P}$ under an element of the rotation group $\mathbb Z_3$.
Since $\tilde{\gamma}$ is a geodesic, its tangent vector extends as a parallel vector field on
$\Sigma_0$. Since the non-identity elements of $\mathbb Z_3$ do not fix this parallel vector
field it follows that the curve $\tilde{\gamma}$ is a closed geodesic of $\Sigma_0$ and
thus has length at least $2\sqrt{3}$, the shortest translation distance in the lattice $\Lambda$.
The second case we consider is the case when $\gamma$ intersects precisely one of the
cone point neighborhoods. In this case we again lift $\gamma$ beginning at the neighborhood
of the cone point to a curve $\tilde{\gamma}$ in $\Sigma_0$. Since there is only one
lift of the cone point to $\Sigma_0$ it follows that the initial and final endpoint of the curve 
$\tilde{\gamma}$ lie in a small neighborhood of the same point, and since the minimum
translation distance in $\Lambda$ is $2\sqrt{3}$ we see that the length of $\tilde{\gamma}$ 
must be at least $2\sqrt{3}-\delta$ if
the neighborhood is chosen small enough. Finally we consider the case in which the curve
$\gamma$ intersects more than one cone neighborhood. In this case the length of $\gamma$
is bounded below by $4-\delta>2\sqrt{3}$ if the neighborhoods are chosen small. This is because
the distance between the cone point on $\Sigma$ is equal to $2$ and the curve must have
at least two segments extending from one cone point neighborhood to another since it is a 
closed curve. This completes the proof of Lemma \ref{triangle.geodesic.lemma}.

\end{proof}

We can now prove the main theorem of this section.

\begin{theo}
There exists a complete, asymptotically flat $(M^3, g)$ with non-negative scalar curvature, that is precisely Schwarzschild outside a compact set, whose unique horizon is totally geodesic, strictly stable, and satisfies $\ell(\partial M^3, g) > 4 \, \pi \, m_{\operatorname{ADM}}(M^3, g)$.
\end{theo}
\begin{proof}
In view of Theorem \ref{main.theorem} and the remarks immediately following it, all we need to do is construct a metric $g \in \sM_+$ on $\SS^2$ for which
\[ \ell(\SS^2, g) > 4 \, \pi \, m_B(\SS^2, g, H = 0) = 4 \, \pi \, \sqrt{\frac{\area(\SS^2, g)}{16 \, \pi}} \Leftrightarrow \frac{\ell(\SS^2, g)^2}{\area(\SS^2, g)} > \pi \text{.} \]

Let $\delta$ be a small positive number and let $\Sigma_\delta$ be the two sphere with metric constructed in Lemma \ref{triangle.geodesic.lemma}. We then have  
$\area(\Sigma_\delta) \leq 2 \sqrt{3} +\delta$ and that
$\ell(\Sigma_\delta) \geq 2 \sqrt{3}-\delta$. Thus,
\[ \frac{\ell(\Sigma_\delta)^2}{\area(\Sigma_\delta)} \geq \frac{(2\sqrt{3}-\delta)^2}{2\sqrt{3} +\delta} > \pi \text{ when } 0 < \delta \ll 1 \text{.} \]
Since the surfaces $\Sigma_\delta$ are smooth spheres with non-negative curvature, the result 
follows by applying Theorem \ref{main.theorem} to $\Sigma_\delta$, for $0 < \delta \ll 1$.
\end{proof}

% Appendix

\appendix

\section{Smooth families of eigenfunctions of $-\Delta + K$}

\label{app:eigenfunction.smoothness}

A key construction in this article is that of families of eigenfunctions $u(t, \cdot)$ of self adjoint second order linear elliptic operators that depend smoothly on a background metric $g(t)$ that itself depends smoothly on the parameter $t$. Since we're using the functions $u(t, x)$ as a warping factor on a Riemannian manifold, it is important that we know $(t, x) \mapsto u(t, x)$ is smooth. In this appendix we derive this in a general context.

\begin{lemm}
Let $\Omega \subset \RR^d$ be open, and suppose $\Omega \ni p \mapsto g(p)$ is a smooth family of metrics on a compact manifold $\Sigma^n$. Let $g \mapsto L_g$ be a family of self adjoint second order linear elliptic operators whose coefficients depend smoothly on $g$, and whose first eigenvalue $\lambda(g)$ has a one-dimensional eigenspace (with boundary conditions, if necessary). Then $\lambda : \Omega \to \RR$ is smooth, and there exists a smooth map $u : \Omega \times \Sigma \to \RR$ so that $u(p, \cdot)$ is an eigenfunction of $L_{g(p)}$ with eigenvalue $\lambda(g(p))$ (and the same boundary conditions as before if necessary).
\end{lemm}

\begin{rema}
This will be applied with $\Omega$ being a bounded time interval, $L_g = -\Delta_g + K_g$, and $\Sigma^2 \approx \SS^2$. The first eigenspace of $L_g$ is one-dimensional by the maximum principle.
\end{rema}

\begin{proof}
It suffices to prove this for $\Omega$ any sufficiently small neighborhood of the origin in $\RR^d$.

For $\varepsilon > 0$, $\delta > 0$ small enough, $\sigma(L_{g(p)}) \cap B_\delta(\lambda(0)) = \{ \lambda(p) \}$ for all $|p| < \varepsilon$ since we know that spectra depend continuously on the operators and that they're all contained within $\RR$ by self adjointness. For all such $p$ the following contour integral over $\gamma = \partial B_\delta(\lambda(0)) \subset \CC$ makes sense:
\[ \Pi(p, u) = \frac{1}{2\pi i} \int_{\gamma} (\zeta \, I - L_{g(p)})^{-1} \, u \, d\zeta \text{.} \]
This is the spectral projection map onto $\ker (L_{g(p)} - \lambda(g(p)) \, I)$, for all $|p| < \varepsilon$. By elliptic theory this is a map $\Pi : B_\varepsilon(0) \times C^\infty(\Sigma^n; \CC) \to C^\infty(\Sigma^n; \CC)$ that is in fact smooth as a map between Fr\'{e}chet spaces \cite{hamilton-inverse-function-theorem}.

If $u_0 \in \ker (L_{g(0)} - \lambda(0) \, I) \cap C^\infty(\Sigma^n; \CC)$ is $u_0 \not \equiv 0$ then for $p$ small enough so that $\Pi(p, u_0) \not \equiv 0$ the maps $\Pi(p, u_0)$ are all eigenfunctions of $L_{g(p)}$. This family is smooth as a map $B_\eta(p) \times C^\infty(\Sigma^n; \CC) \to C^\infty(\Sigma^n; \CC)$, which in turn yields (the weaker condition of) smoothness as a map $B_\eta(p) \times \Sigma^n \to \CC$. Rescaling smoothly in $p$ (the eigenspaces are one-dimensional) we can get also a smooth real-valued family $B_\eta(p) \times \Sigma^n \to \RR$.
\end{proof}

\bibliographystyle{amsalpha}

\providecommand{\bysame}{\leavevmode\hbox to3em{\hrulefill}\thinspace}
\providecommand{\MR}{\relax\ifhmode\unskip\space\fi MR }
\providecommand{\MRhref}[2]{%
  \href{http://www.ams.org/mathscinet-getitem?mr=#1}{#2}
}
\providecommand{\href}[2]{#2}

\end{document}